\documentclass[a4paper,11 pt,twoside,reqno]{amsart}

\usepackage[utf8]{inputenc}
\usepackage[english]{babel}
\usepackage{times}
\usepackage{amsmath}
\usepackage{amsfonts}
\usepackage{amssymb}
\usepackage{amsmath}
\usepackage{amsthm}
\usepackage{graphicx}
\usepackage{array}
\usepackage{color}
\usepackage{mathrsfs}
\usepackage{hyperref}
\usepackage{esint} 
\usepackage{tikz}
\usepackage{ upgreek }
\usepackage{enumitem}

\definecolor{darkgreen}{rgb}{0.1,0.7,0.1}
\definecolor{darkred}{rgb}{0.7,0.1,0.1}

\newtheorem{theorem}{Theorem}
\newtheorem{lemma}{Lemma}[section]
\newtheorem{proposition}[lemma]{Proposition}

\newtheorem{conjecture}[lemma]{Conjecture}
\newtheorem{remark}[lemma]{Remark}

\newcommand{\eps}{\varepsilon}

\newcommand\symb[2][\bf]{{\mathchoice{\hbox{#1#2}}{\hbox{#1#2}}%
        {\hbox{\scriptsize#1#2}}{\hbox{\tiny#1#2}}}}

\def\R{{\symb R}}
\def\N{{\symb N}}

\def\P{{\symb P}}

\renewcommand{\P}{\mathbb{P}}
\newcommand{\E}{\mathbb{E}}

\newcommand{\cE}{\mathcal{E}}

\newcommand{\gl}{\lambda}

\hypersetup{
citecolor=blue,
colorlinks=true, 
breaklinks=true, 
urlcolor= blue, 
linkcolor= black, 
bookmarksopen=true, 
pdftitle={The spectrum of the stochastic Bessel operator at high temperature}, 
}

\begin{document}

\title{The spectrum of the stochastic Bessel operator at high temperature}

	\author{Laure Dumaz}
\address{CNRS \& Department of Mathematics and Applications, \'Ecole Normale Sup\'erieure (Paris), 45 rue d’Ulm, 75005 Paris, France}
\email{laure.dumaz@ens.fr}
\author{Hugo Magaldi}
\address{}
\email{magaldi.hugo@gmail.com}

\date{\today}

\maketitle

\begin{abstract}
	Ram\'irez and Rider \cite{RamirezRider2009} established that the hard edge of the spectrum of the $\beta$-Laguerre ensemble converges, in the high-dimensional limit, to the bottom of the spectrum of the stochastic Bessel operator. Using stochastic analysis tools, we prove that, in the high-temperature limit ($\beta \to 0$), the rescaled eigenvalue point process of this operator converges to a non-trivial limiting point process. This limit is characterized by a family of coupled diffusions and differs from a Poisson point process due to its interaction with the hard edge. Exploiting this diffusion characterization, we establish exact large deviation asymptotics for the largest eigenvalues. Furthermore, for an explicit range of the parameters, we relate this limiting process to the finite-$n$ $\beta$-Laguerre ensemble, conjecturing an exact distributional match with the infinite sum of its independent exponential gaps. As a byproduct of our analysis, we also formulate a conjecture regarding an explicit integral formula for the probability that a reflected Brownian motion with a constant drift hits an affine line, generalizing a formula of Salminen and Yor \cite{SalminenYor}.
\end{abstract}

\tableofcontents

\section{Introduction}

\subsection*{$(\beta,a)$-Laguerre ensembles.}

In this paper, we study the high-temperature limit of the spectrum of the Stochastic Bessel Operator (SBO). Its finite-dimensional counterpart is the famous $(\beta,a)$-Laguerre ensemble, a probability measure on the set of ordered points in $\R_+^n$ with density
\begin{align}\label{densityLaguerre}
	\frac{1}{Z_{n,\beta,a}} \prod_{i < j} |\gl^{(n)}_i - \gl^{(n)}_j|^\beta \times \prod_{k=1}^{n} (\gl^{(n)}_k)^{\frac{\beta}{2}(a+1) -1} e^{-\frac{\beta}{2} \gl^{(n)}_k} 1_{0 < \gl^{(n)}_1 < \cdots  <\gl^{(n)}_n }\,,
\end{align}
where $a > -1$, $\beta >0$ and $Z_{n,\beta,a}$ is the normalization constant.

For the specific values $\beta = \{1,2,4\}$ and integer $a \in \N$, this joint law describes the eigenvalue distribution of Wishart matrices. These matrices take the form $X X^\intercal$, where $X$ is a rectangular matrix of size $n \times (n+a)$ with independent real ($\beta = 1$), complex ($\beta = 2$), or quaternionic ($\beta = 4$) Gaussian entries of mean zero and variance $1$.  Note that the matrix $(n+ a)^{-1} X X^\intercal$ represents the sample covariance matrix of $n$ independent data vectors, each of them containing $n+a$ i.i.d. standard Gaussian entries.

\medskip

Dumitriu and Edelman \cite{DumitriuEdelman} constructed a tridiagonal random matrix model for general $\beta >0$, known as the $\beta$-Wishart ensemble, whose eigenvalues are distributed according to the density \eqref{densityLaguerre}. The matrix is written as $L_{\beta,a} L_{\beta,a}^\intercal$, where
\begin{align}
	L_{\beta,a} 
	= \frac{1}{\sqrt{\beta}} \begin{pmatrix} \chi_{(a+n)\beta} &  \chi_{(n-1)\beta} & & &  \\ 
	&	\chi_{(a+n-1)\beta} & \chi_{(n-2)\beta} & & &  \\
	&&&&&\\
	 &&  \ddots  & \ddots & \label{matrixLbetaa}\\ 
	 &&&&&\\
	 & & & \chi_{(a+2)\beta} & \chi_{\beta}  \\
	 &&&&  \chi_{(a+1)\beta}
	\end{pmatrix}\,
\end{align}
The entries $\chi_r$ appearing in the matrix are independent $\chi$-distributed random variables with the indicated parameter $r$. While the construction for the classical cases $\beta = 1,2,4$ relies on bidiagonalizing the initial Wishart matrices via Householder reflections, the resulting tridiagonal matrix naturally generalizes to all values of $
\beta >0$, providing a matrix model to study \eqref{densityLaguerre} for general $\beta$.

\medskip

Physically, the particles distributed according to the law \eqref{densityLaguerre} can be interpreted as a log-gas at inverse temperature $\beta$ confined to the half-line $\mathbb{R}_+$ by a log-Gamma potential \cite{forrester2010}. The $n$ particles are positively charged and repel each other. The interaction with the hard edge at the origin can be modeled by an additional particle fixed at $0$. The effective sign of this ghost charge depends on the parameters: when $a \geq 2/\beta - 1$, the hard edge is repulsive (positively charged), whereas when $a < 2/\beta -1$, the hard edge becomes attractive (negatively charged), causing particles to accumulate near the origin. To rigorously capture the microscopic behavior of these particles near the hard edge and decouple the system from finite-$n$ effects, we turn to the continuous limit.

\subsection*{The Stochastic Bessel Operator.} The simple tridiagonal structure of the matrix \eqref{matrixLbetaa} provides a powerful tool for deriving several important statistical properties. In particular, as noticed by Edelman and Sutton \cite{edelman_sutton2007}, it acts as a discrete differential operator, providing a way to take its limit at the level of self-adjoint differential operators. Ram\'irez and Rider implemented this approach for the matrix \eqref{matrixLbetaa}, and proved in \cite{RamirezRider2009} that the point process of the smallest rescaled eigenvalues $n \gl_i^{(n)}$ has a limit as $n \to \infty$. These rescaled eigenvalues converge towards the smallest eigenvalues of the stochastic Bessel operator (SBO).
For all $\beta >0$ and $a > -1$, the SBO is defined as
\begin{align*}
	\mathfrak{G}_{\beta,a} = - \exp \Big((a+1) x + \frac{2}{\sqrt{\beta}} b(x) \Big) \cdot \frac{d}{dx} \Big(\exp\Big(- a x - \frac{2}{\sqrt{\beta}} b(x)\Big) \frac{d}{dx}\Big)\,,
\end{align*}
where $b$ is a standard Brownian motion. It formally corresponds to the differential operator
\begin{align*}
	e^x \Big(- \frac{d^2}{dx^2} + (a + \frac{2}{\sqrt{\beta}} b'(x)) \frac{d}{dx}\Big)\,.
\end{align*}
This operator acts as the infinitesimal generator (modulo the time change $e^x$) of a diffusion evolving in a white noise potential given by $b'$. It is defined on a suitable domain of $L^2(\R_+,m)$ where $m(dx) = \exp(-(a+1) x - \frac{2}{\sqrt{\beta}} b(x)) dx$ with Dirichlet boundary condition at the origin.

The operator $\mathfrak{G}_{\beta,a}$ has a discrete spectrum of simple eigenvalues:
\begin{align*}
	0 < \lambda^{\infty}_{\beta,a}(1)<\lambda^{\infty}_{\beta,a}(2) < \cdots\,.
\end{align*}

Moreover, applying Riccati’s map to the solution of the eigenvalue equation with Dirichlet boundary condition at $0$, Ram\'irez and Rider introduced a family of diffusions characterizing the point process of SBO's eigenvalues which will be instrumental in our proof (see paragraph \ref{subsec:RescaledDiff} below).

\subsection*{Convergence in the high-temperature limit ($\beta \to 0$).} 
We focus here on the small $\beta$ regime, often referred to as the high-temperature limit. Specifically, we investigate the microscopic scale---where the spacing between particles is of order $1$---within the universal scaling limits $N \to \infty$ of finite-$n$ $\beta$ ensembles. The high-temperature regime of finite-$n$ $\beta$-ensembles has been the subject of several recent studies. In the double limit $n \to \infty$ and $\beta \to 0$ such that $\beta n$ converges to a constant, the global fluctuations and macroscopic limits were studied in  \cite{allez2012invariant}, \cite{nakano2018gaussian}, \cite{trinh2021beta}, \cite{hardy2021clt}. Under the same scalings, at the microscopic level, convergence to a Poisson point process has been established in \cite{pakzad2019poisson}, \cite{nakano2020poisson}. This decorrelation phenomenon extends to the universal scaling limits---namely Sine$_\beta$ in the bulk and Airy$_\beta$ at the edge. For these, it was shown in  \cite{allez_dumaz2014_tw, allez_dumaz2014_sine, dumaz_labbe2022},  that under an appropriate rescaling, the eigenvalue point process converges towards a Poisson point process.

For SBO's eigenvalues, when $\beta$ tends to $0$, the smallest eigenvalues approach the hard edge at $0$ at an exponential rate due to the strong attraction at the origin. Specifically, the smallest eigenvalues are of order $\exp(-\Theta(1)/\beta)$. To obtain a non-trivial limit, we consider the rescaled eigenvalues:
\begin{align*}
	\beta \ln(1/\lambda^{\infty}_{\beta,a}(i)),\quad i \geq 1\,.
\end{align*}
Under this rescaling, the spacing between the first points is of order $O(1)$. This transformation reverses the ordering, meaning the rescaled point process is bounded from above. Furthermore, the inverse mapping $\mu \mapsto \exp(-\mu/\beta)$ has a sharp transition near $\mu=0$. In the small $\beta$ limit, for any fixed $\mu >0$, there are only finitely many eigenvalues below $ \exp(-\mu/\beta)$. Conversely, for $\mu \leq 0$, the number of eigenvalues explodes as $\beta \to 0$, as we are no longer exponentially close to the hard edge.

\medskip

Let us first state our convergence theorem:
\begin{theorem}[Convergence of the eigenvalues]\label{theoCV}
As $\beta \to 0$, the rescaled eigenvalue point process of the SBO $(\beta \ln(1/\lambda^{\infty}_{\beta,a}(i)),\; i \geq 1)$ converges in law towards a random simple point process on $\R_+$ with an accumulation point at $0+$. 
\end{theorem}
We denote this point process, ordered in a decreasing order, by $(\mu^{\infty}_a (i),\; i \geq 1)$.
The convergence holds for a well-chosen topology of Radon measures on $\R_+$, corresponding to a left-vague/right-weak topology.

\subsection*{Characterization of the limiting point process via coupled diffusions.}
We characterize the limiting process using a family of diffusions $(r^{(a)}_\mu)_{\mu > 0}$, a construction similar in spirit to the one used for the SBO eigenvalues. 

\subsubsection*{Construction of the diffusions.}
This construction relies on a family of Brownian motions with drift reflected at $0$. Recall that a reflected Brownian motion with drift $\delta$ starting at $0$, denoted $R^{(\delta)}$, admits the Skorokhod representation:
\begin{align*}
R^{(\delta)}(t)	= B(t) + \delta \, t + L(t)\,,
\end{align*}
where $B$ is a Brownian motion and $L(t) = \sup_{0 \leq s \leq t}  \max(-B(s) - \delta\, s,0)$. 

Recall the classical Skorokhod reflection problem: for any continuous path $y \colon [0, \infty) \to \mathbb{R}$ with $y(0) \ge 0$, there exists a unique pair of continuous functions $(x, \ell)$ such that
\begin{itemize}
	\item $x(t) = y(t) + \ell(t) \ge 0$ for all $t \ge 0$,
	\item $\ell$ is non-decreasing with $\ell(0) = 0$,
	\item $\int_0^t x(s) \, d\ell(s) = 0$ for all $t \ge 0$.
\end{itemize}
For a detailed proof of existence and uniqueness of this mapping, see e.g. \cite[Lemma 3.6.14]{karatzas} or \cite[Chapter IX, Theorem 2.3]{yor}.
In our context, we apply the Skorokhod mapping to the Brownian motion with drift $\delta$, which yields the pair of the reflected process $R^{(\delta)}$ and $L$ (corresponding to the local time up to a constant factor).

Let us first define the diffusion $r^{(a)}_\mu$ for a \emph{single fixed parameter} $\mu >0$. To ease the presentation, we will drop in the following the superscript $a$. We define the  \emph{critical line} as $c_{\mu}(t) = \mu + t/4$ for $t \geq 0$. The process is constructed recursively. We refer to Figure \ref{fig:1diff} for an illustration.
\begin{itemize}
	\item \textbf{Initialization:} Set $r_{\mu}(0) = 0$.
	\item \textbf{Phase $-$:} The process evolves as a reflected Brownian motion with drift $-a/4$ until it hits the critical line $c_\mu$. Let $\xi^-_{\mu}(1)$ denote this hitting time.
	\item \textbf{Reset:} Immediately upon hitting $c_\mu$,	the process jumps to $0$.
	\item \textbf{Phase $+$:} It then evolves as a reflected Brownian motion with drift $(a+1)/4$ until it hits $c_\mu$ again. Let $\xi^+_{\mu}(1)$ denote this second hitting time.
	\item \textbf{Recursion:} This cycle repeats, alternating between the drifts $-a/4$ and $(a+1)/4$ after each hit of the critical line. We denote the hitting times as $\xi^-_{\mu}(i)$ and $\xi^+_{\mu}(i)$ for $i \geq 2$.
\end{itemize}
\begin{figure}[!h]
	\includegraphics[width=12cm]{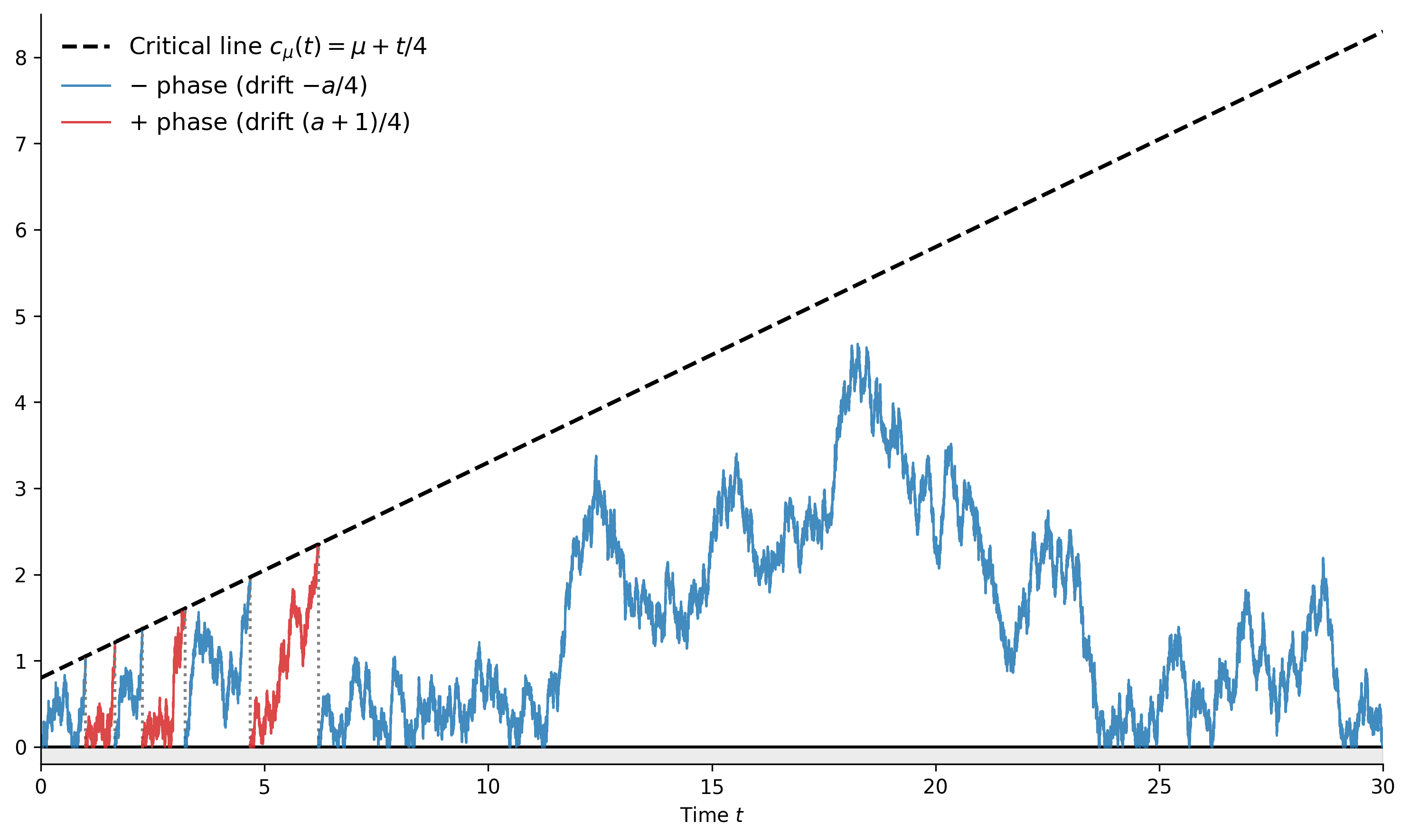}
	\caption{Simulation of the diffusion $r_\mu$ for $\mu = 0.8$ and $a=1$.}\label{fig:1diff}
\end{figure}

    Since the critical line $t \mapsto \mu + t/4$ grows linearly, the probability that the reflected Brownian motion with drift reaches it in finite time is strictly less than one for at least one of the two phases (see Remark \ref{remDriftCL} below). Consequently, the total number of completed cycles is finite almost surely (it is stochastically dominated by a geometric distribution).

    \begin{remark}[Drifts versus critical line]\label{remDriftCL}
	A reflected Brownian motion with drift $\delta$ hits the critical line $t \mapsto \mu + t/4$ almost surely if and only if $\delta \geq 1/4$.
	\begin{itemize}
		\item For the phase with drift $-a/4$: Since we assume $a > -1$, this drift is always strictly less than $1/4$.
		\item For the phase with drift $(a+1)/4$: If $a \geq 0$, the drift is $\geq 1/4$: a hit holds almost surely. If $a \in (-1,0)$, the drift is less than $1/4$, so the hit occurs with probability strictly less than $1$.
	\end{itemize}
	Thus, for $a \in (-1,0)$, both phases have a probability strictly less than $1$ to hit the critical line. 
\end{remark}

We now define the family of diffusions $(r_\mu)_{\mu}$ simultaneously for all $\mu >0$ by coupling them through a single driving Brownian motion $B$.

For any drift $\delta$ and starting time $s \geq 0$, let us define the path $R^{(\delta,s)}$ starting at $0$ at time $s$, driven by $B$:
\begin{align*}
	R^{(\delta, \, s)}(t) = 	Y^{\delta,s}(t) + \sup_{s \leq u \leq t} \max(0,-Y^{\delta,s}(u))\, , \qquad t \geq s \,,
\end{align*}
where $Y^{\delta,s}(t) = \delta (t-s) + (B(t)-B(s))$ is the non-reflected path.

Given this coupling of reflected Brownian motion, we define the family $(r_\mu)_\mu$ as follows. Let $\xi^+_{\mu}(0) :=0$. For all $i \geq 1$, we set
\begin{align*}
	r_\mu(t) := R^{(-a/4,\, \xi^+_{\mu}(i-1))}(t),\qquad \mbox{for }t \in [\xi^+_{\mu}(i-1),\xi^-_{\mu}(i))\,,
\end{align*}
where the hitting time is
\begin{align*}
	\xi^-_{\mu}(i) := \inf\{t \geq \xi^+_{\mu}(i-1),\; R^{(-a/4, \, \xi^+_{\mu}(i-1))}(t)  = c_\mu(t)\}\,.
\end{align*}
Subsequently,
\begin{align*}
	r_\mu(t) := R^{((a+1)/4, \, \xi^-_{\mu}(i))}(t),\qquad \mbox{for }t \in [\xi^-_{\mu}(i),\xi^+_{\mu}(i))\,,
\end{align*}
where
\begin{align*}
	\xi^+_{\mu}(i) := \inf\{t \geq \xi^-_{\mu}(i),\;  R^{((a+1)/4, \xi^-_{\mu}(i))}(t)  = c_\mu(t)\}\,.
\end{align*}

An illustration of the coupled diffusions $r_{\mu}$ for two different values of $\mu$ can be found in Figure \ref{fig:2diff}.
\begin{figure}[!h]
	\includegraphics[width=13cm]{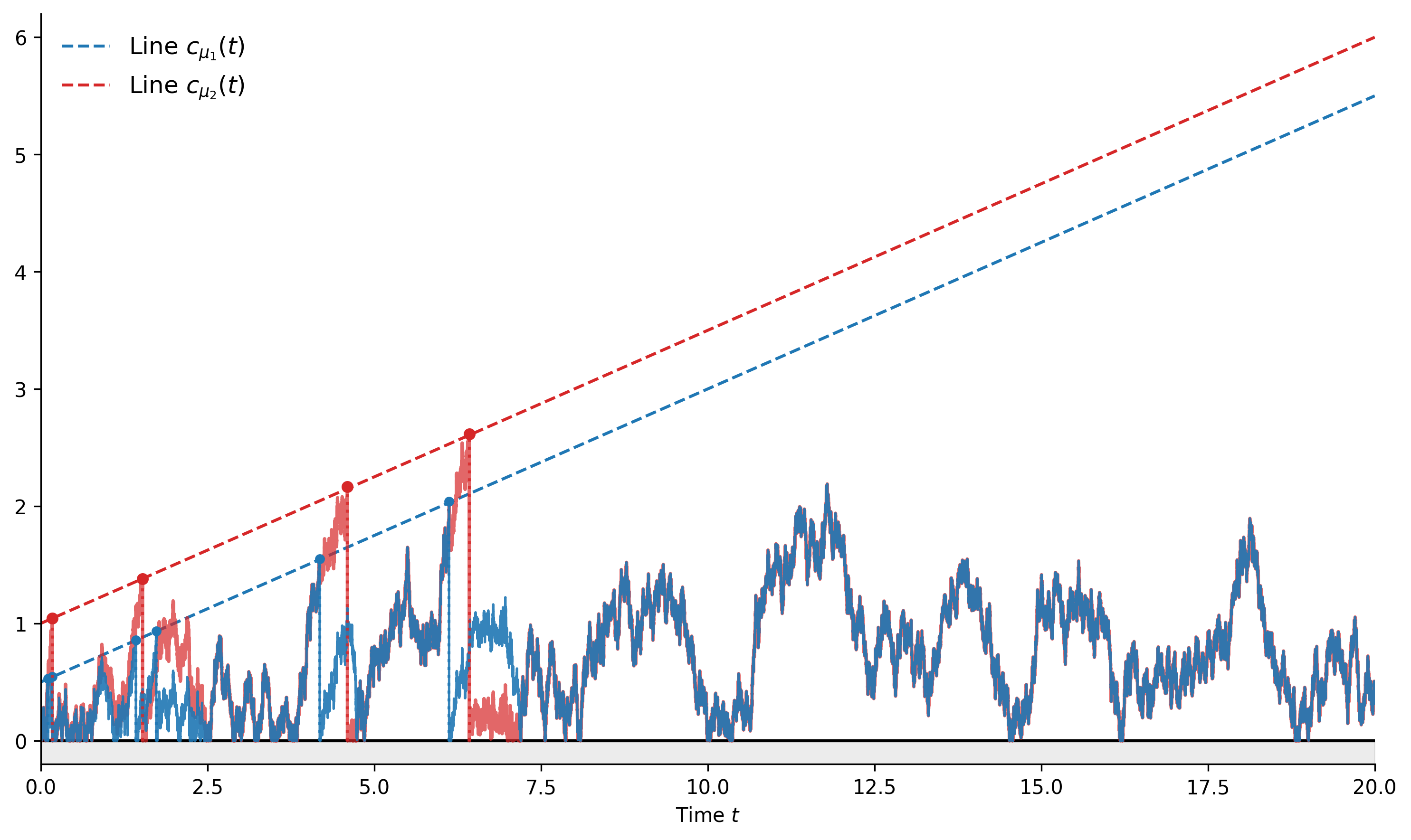}
	\caption{Simulation of two diffusions $r_\mu$ for $\mu = 0.5$ (blue) and $\mu =1$ (red), with $a=1$ (color online).}\label{fig:2diff}
\end{figure}

\subsubsection*{Characterization of the limiting point process}
  We associate a random counting measure to the completion times of the $+$ phases only
\begin{align*}
	\nu_\mu := \sum_{i \geq 1} \delta_{\xi^+_\mu(i)}\,.
\end{align*}
The mass $\nu_\mu(\mathbb{R}_+)$ corresponds to the total number of full cycles completed by the diffusion $r_\mu$. It is easy to see that the total mass is finite almost surely: as observed in Remark \ref{remDriftCL}, at least one of the two phases has a drift strictly less than $1/4$. Since the probability for a reflected Brownian motion to hit an affine line growing strictly faster than its drift is strictly less than $1$, the number of points of $\nu_\mu$ is stochastically dominated by a geometric random variable.

As $\mu$ increases, the barrier $c_\mu$ shifts upwards, making it harder to hit. Thus, the map $\mu \mapsto \nu_\mu(\mathbb{R}_+)$ is non-increasing almost surely and takes values in $\mathbb{N} \cup \{0\}$.

\medskip

We define the limiting point process $\mathcal{M}_a$ on $(0, \infty)$ via its cumulative count: for any $\mu > 0$, the number of points in $\mathcal{M}_a$ greater than $\mu$ is given exactly by the number of cycles of $r_\mu$:
\begin{align*}
	\mathcal{M}_a([\mu, \infty)) := \nu_\mu(\mathbb{R}_+)\,.
\end{align*}

\medskip

We can now state the characterization of the limiting point process $\sum_{ i\geq 1} \delta_{\mu_a^{\infty}(i)}$ defined in Theorem \ref{theoCV}:

\begin{theorem}[Characterization of the limiting point process]\label{theoCharact}
	The limiting point process $\sum_{ i\geq 1} \delta_{\mu_a^{\infty}(i)}$ has the same distribution as $\mathcal{M}_a$.
\end{theorem}

\subsection*{Comments and further results.} This characterization allows us to compute several statistics for the limiting point process.

\subsubsection*{Largest eigenvalue}
 For $a \geq 0$, the distribution of the largest eigenvalue is related to the hitting time of an affine boundary and takes a simple form:
\begin{theorem}[First eigenvalue limit for $a \geq 0$]
Let $R^{(-a/4)}(t)$ be a reflected Brownian motion (at $0$) with drift $-a/4$. Then the probability that the largest point $\mu^{\infty}_a (1)$ exceeds $\mu$ equals the probability that $R^{(-a/4)}$ hits the affine line $t \mapsto \mu + t/4$ in a finite time.
\end{theorem}

For $a = 0$, using Corollary 5 of \cite{SalminenYor}, this probability is given by
\begin{align}\label{RBM0hit}
	\P(\exists t \geq 0,\; R^{(0)}(t) \geq \mu + t/4) = 2  \sum_{k \geq 1} (-1)^{k-1} \exp(-  \mu k^2/2)\,.
\end{align}
We have a conjecture for an explicit expression of the largest eigenvalue for all $a >0$, see the equation \eqref{expressiondensityfirst}.

\subsubsection*{Large $a$ limit conjecture}
Moreover, we expect that one recovers a Poisson point process in the large $a$ limit. This is due to the fact that the coupled diffusions will have to move against a strong drift to hit the critical line, and therefore will typically hits it very quickly when the hitting time is finite: this compensates the growth of the critical line. This expected transition to Poissonian statistics mirrors the behavior observed in the high-temperature bulk or edge \cite{allez_dumaz2014_sine, allez_dumaz2014_tw, dumaz_labbe2022}. This can be linked as well to the transition of the Stochastic Bessel Operator to the Stochastic Airy Operator as $a \to \infty$ proved in \cite{dumaz_li_valko2021}. We formalize this as a conjecture.
\begin{conjecture}
	When $a \to \infty$, the rescaled point process $\mathcal{M}_a$ converges towards a Poisson point process.
\end{conjecture}

\subsubsection*{Finite-$n$ scaling}
Finally, consider the finite-$n$ $(\beta,a)$-Laguerre ensemble. When $\beta \to 0$, we obtain the following density for the rescaled particles  $\mu^{(n)}_k = \beta \ln(1/\gl^{(n)}_k)$. 
\begin{align*}
	\frac{1}{Z_{n,a}} \prod_{i=1}^n \exp(- (i-1 + (a+1)/2) \mu^{(n)}_i) 1_{\mu^{(n)}_1 > \mu_2^{(n)} > \cdots > \mu_n^{(n)} >0} \,.
\end{align*}
This point process is naturally described through its gaps $g_j := \mu^{(n)}_{j} - \mu^{(n)}_{j+1}$, for $j \in \{1,\dots,n\})$ (where $\mu^{(n)}_{n+1} := 0$). These gaps are $n$ independent random variables, where $g_j$ follows an exponential distribution with parameter $j(j+a)/2$. Therefore, one can construct the point process from the smallest point to the largest by successively adding independent gaps. 
Reconstructing the points from the gaps requires working backwards: the largest point is $\mu^{(n)}_{1} = \sum_{j=1}^n g_{j}$, and the subsequent points are given by the tail sums $\mu^{(n)}_{i} = \sum_{j=i}^n g_{j}$.

One can check that the largest points of this point process have a well-defined limit as $n \to \infty$. This limit can be constructed recursively: first generate an infinite sequence of independent gaps $(g_j , \; j \geq 1)$ following exponential distributions with parameters $j(j+a)/2$ . The largest point is then given by the infinite sum $\hat{\mu}^\infty_a(1) :=  \sum_{j=1}^\infty g_{j}$, and subsequent points are found by successively removing gaps: 
\begin{align}
	\hat{\mu}^\infty_a(k) :=  \sum_{j=k}^\infty g_{j} \,.\label{expressionhatmu}
\end{align}

Taking the double limit ---first as $\beta \to 0$ and then as $n \to \infty$--- we see that the repulsion does not disappear but its nature changes. It is no longer local as the gaps are given by exponential laws (which allow for small distances between the points), rather, the repulsion induces a deterministic shift in the exponential rates.

It turns out that in the case $a=0$, one can show that the largest point $\mu^\infty_0(1)$ of $\mathcal{M}_0$ and the largest point $\hat{\mu}^\infty_0(1)$ coincide.
\begin{proposition}[Exact matching for $a=0$]
The distribution of $\mu^\infty_0(1)$ and $\hat{\mu}^\infty_0(1)$ are equal.
\end{proposition}
\begin{proof}
	The distribution of $\hat{\mu}^\infty_0(1)$ is that of an infinite sum of independent exponential random variables of parameter $j^2/2$. Its Laplace transform is given by
	\begin{align*}
		\E[\exp(- \theta \, \hat{\mu}^\infty_0(1) )] = \prod_{j=1}^\infty \frac{1}{(1+ \frac{2\theta}{j^2})}\,.
	\end{align*}
	This product has an explicit simple expression in this case:
	\begin{align*}
		\prod_{j=1}^\infty \frac{1}{1+ \frac{2\theta}{j^2}} = \frac{\pi \sqrt{2 \theta}}{\sinh(\pi \sqrt{2 \theta})}\,.
	\end{align*}
	Inverting this Laplace transform, we obtain for the density:
	\begin{align*}
		\sum_{k=1}^{\infty} (-1)^{k-1} k^2 \exp(-k^2 t/2)\,,
	\end{align*}
	which integrates to give the tail probability \eqref{RBM0hit}.
\end{proof}

We conjecture that this correspondence remains valid for $a>0$ across the entire point process. If true, this would imply a striking diffusion representation for \eqref{expressionhatmu} --- a connection that is, to the best of our knowledge, entirely unexpected. As a byproduct, establishing this conjecture would provide an exact integral formula for the probability that a reflected Brownian motion with negative drift hits an affine line.
\begin{conjecture}[Matching for all $a \geq 0$]\label{conjageq0}
For all $a \geq 0$, the joint distribution of $(\hat{\mu}_a^{\infty}(i),\; i\geq 1)$ and $({\mu}_a^{\infty}(i),\;i\geq 1)$ are equal. As a corollary, using the explicit expression for the density of $\hat{\mu}_a^{\infty}(1)$ given by:
\begin{align}
	\sum_{j=1}^{\infty} (-1)^{j-1} \frac{j(2j+a)}{2} \binom{j+a}{j} e^{-\frac{j(j+a)}{2}x}\,.\label{expressiondensityfirst}
\end{align}
This matching yields the following explicit integral expression for the probability that a reflected Brownian motion with drift $-a$ hits the affine line $\mu +  b t$:
\begin{align} \P\big(\exists t \geq 0, \; R^{-a}(t) \geq \mu + b t\big) = \sum_{j=1}^\infty (-1)^{j-1} \left[ \binom{j + a/b}{j} + \binom{j + a/b - 1}{j-1} \right] e^{-2 \mu j (jb + a)}\,. \label{conj_reflectedBMdrift}
	\end{align}
\end{conjecture}
A natural pathway to prove the integral expression \eqref{conj_reflectedBMdrift} would be to generalize the approach of Salminen and Yor to a reflected Brownian motion with negative drift.

\begin{remark}[Breakdown in the regime $a \in (-1, 0)$]
	We emphasize that Conjecture \ref{conjageq0} is restricted to $a \geq 0$. For $a \in (-1, 0)$, this exact matching is no longer true, as the asymptotic tails do not coincide, as we will see below. This provides an explicit example where the limits do not commute, illustrating why the exchange of such limits is usually a quite subtle issue in the literature.
\end{remark}

	\medskip
	
\subsubsection*{Non-poissonian limit}
Thanks to its interpretation with reflected Brownian motions, one can compute the asymptotic of finding more than $k$ points above some large level $\mu$ when $\mu \gg1$, which has an explicit expression. 
\begin{proposition}[Asymptotics for the largest $k$ eigenvalues]\label{propo:asympkexplo}
	Fix $a \geq -1$. For any fixed integer $k \geq 1$, we have as $\mu \to \infty$,
	\begin{align*}
		\P[\mathcal{M}_a[\mu,\infty) \geq k] = \exp\Big(- \frac{k(|a|+k)}{2} \mu (1+ o(1))\Big)\,.
	\end{align*}
\end{proposition}

The motivation to prove such an asymptotic is twofold. Firstly, it will easily provide a proof that the limiting point process is not a Poisson point process as we will see just below. Secondly, it provides strong evidence that the point process $(\mu^{\infty}_{a}(k),\;k\geq 1)$ coincides with $(\hat \mu^{\infty}_{a}(k),\;k\geq 1)$ for $a \geq 0$ since their asymptotic decay rates are identical. As mentioned above, it also confirms that the limits \emph{do not commute} when $a \in (-1,0)$, giving an example where the infinite universal limit captures ``more repulsion'' than the finite-$n$ limit.

An easy way to check that our limiting point process is not a Poisson point process is to look at the behavior of 
\begin{align*}
	\frac{\P[\mathcal{M}_a[\mu,\infty) \geq 2]}{\P[\mathcal{M}_a[\mu,\infty) \geq 1]^2}\,,
\end{align*}
in the large $\mu$ limit, which should be equal to $1/2$ for a Poisson point process. The proposition above directly proves that this is not the case, as indeed, we obtain for $\mu \gg 1$,
\begin{align*}
	\frac{\P[\mathcal{M}_a[\mu,\infty) \geq 2]}{\P[\mathcal{M}_a[\mu,\infty) \geq 1]^2} = \exp(- \mu(1+o(1)) )\,.
\end{align*}

\subsubsection*{Local level repulsion when $a \geq 0$}
Let us conclude with a few words about the \emph{local level repulsion} for $a \geq 0$. Consider the gap between the first and second limiting points, $h_1 := \mu_a^\infty(1) - \mu_a^\infty(2)$. Conditional on the largest eigenvalue being macroscopic, say $\mu_a^\infty(1) \geq \mu_0 >0$, one can investigate the probability that this gap is atypically small, namely $$\P(h_1 \leq \eps \mid \mu_a^\infty(1) \geq \mu_0).$$ 
	
	In classical $\beta$-ensembles, strong local repulsion yields a probability scaling as $O(\eps^{1+\beta})$. Here, however, a back-of-the-envelope computation provided below suggests that this probability scales linearly as $O(\eps)$. This confirms that the strong microscopic repulsion vanishes in the limit, matching the behavior of the independent exponential gaps observed in the finite-$n$ scaling. 
	
	To see this, recall that the first eigenvalue corresponds to $\mu_1 := \mu_a^\infty(1) = \max \{\mu > 0 : \exists t \geq 0,\; r_{\mu}(t) = c_\mu(t)\}$. Let us look at the unconstrained diffusion $r_1 := r_{\mu_1+1}$, which evolves simply as a reflected Brownian motion with drift $-a/4$ without undergoing any resetting procedure (the shift by $+1$ is arbitrary to ensure it never hits its critical line). 
	
	Now, consider the shifted level $\tilde \mu_1 := \mu_1 - \eps$ and its associated path $r_{\tilde \mu_1}$. For the event $\{h_1 \le \eps\}$ to occur, $r_{\tilde \mu_1}$ must complete a full cycle and hit $c_{\tilde \mu_1}$ again. Because the paths are coupled via the same driving Brownian motion, once $r_{\tilde \mu_1}$ completes its short resetting $+$ phase, its trajectory must coincide with $r_1$ (unless when $r_1$ remains trapped within the narrow macroscopic strip $[c_{\mu_1 - \eps}, c_{\mu_1}]$ whose probability is exponentially small). 
	
	This implies that $r_1$ hits its maximal line $c_{\mu_1}$ and then, after a macroscopic amount of time, returns to within a distance $\eps$ of it. By classical path decomposition theorems, the trajectory of a Brownian motion with drift, conditioned on its overall maximum, behaves locally like a Bessel process (with drift). The probability that such a process returns to within a distance $\eps$ of its macroscopic maximum scales exactly as $O(\eps)$.

\subsection*{Acknowledgements} Special thanks are due to Djalil Chafai for many helpful discussions and for his invaluable support throughout this project. L.D. acknowledges the support of ANR RANDOP ANR-24-CE40-3377 and LOCAL ANR-22-CE40-0012.

\section{Strategy of proof}

\subsection{Rescaled diffusions.} \label{subsec:RescaledDiff}

Our approach relies on the characterization of SBO spectrum via a family of coupled  diffusions $(p^{\beta}_\gl,\; \gl \in \R_+^*)$ with initial condition $p^{\beta}_\gl(0)=+ \infty$:
\begin{align*}
	dp^{\beta}_\gl(t) = \frac{2}{\sqrt{\beta}} p^{\beta}_\gl(t) db(t) + \big((a+\frac{2}{\beta})p^{\beta}_\gl(t) - p^{\beta}_{\gl}(t)^2 - \gl e^{-t}\big) dt \,.
\end{align*}
The diffusion $p^{\beta}_\gl$ may explode to $-\infty$; in this case it immediately restarts from $+\infty$.

These diffusions correspond to the Riccati transform $p^\beta_\gl = \psi'/\psi$ of the eigenvalue equation $	\mathfrak{G}_{\beta,a} \psi = \gl \psi$ with initial condition $\psi(0) = 0$, see (1.5) in \cite{RamirezRider2009}.

It is crucial to note that the \emph{same} Brownian motion $b$ drives the whole family of SDEs. It implies important properties such as the monotonicity of the number of explosions of $p^{\beta}_\gl$ (which turns out to be finite). In fact, the number of explosions of $p^{\beta}_\gl$ over $\R_+$ corresponds to $N^{\beta}_\gl$ the counting function of the eigenvalues of the SBO.

We will study the small $\beta$ limit of the family $(p^{\beta}_\gl)$ when $\gl$ is properly rescaled with $\beta$, that is when $\beta \ln (1/\gl)$ is of order $1$. More precisely, we focus on the number of explosion times of $p^{\beta}_\gl$ on $\R_+$.

Let us fix $\mu \in \R$. Notice that when $p^{\beta}_\gl$ reaches $0$, the term in front of the noise vanishes and the drift is negative. It implies that $p^{\beta}_\gl$ never reaches $0$ from below. It is easy to check that the hitting times of $0$ form a discrete point process. The first key idea of our analysis 
is to introduce the rescaled diffusion $q^\beta_\mu(t)$, defined piecewise depending on the sign of $p^{\beta}_{\gl_\beta}(t/(4\beta))$, where $\gl_\beta := \exp(-\mu/\beta)$. When the process is positive, we define $q^-_\mu(t) := - \beta \ln (p^{\beta}_{\gl_\beta}(t/(4 \beta)))$. When it is negative, we define $q^+_\mu(t) := \beta \ln(-p^{\beta}_{\gl_\beta}(t/(4\beta))) + \mu + t/4$.

A simple computation using Itô formula shows that they follow the SDEs:
\begin{align}
	dq^-_\mu &= dW(t) + \frac{1}{4} \Big(- a + \exp(-q^-_\mu(t)/\beta) + \exp(-( t/4 + \mu - q^-_\mu(t))/\beta ) \Big) dt \,, \label{SDEq-}\\
	dq^+_\mu &= dW(t) + \frac{1}{4} \Big((a+1) + \exp(-q^+_\mu(t)/\beta) + \exp(-( t/4 + \mu - q^+_\mu(t) )/\beta ) \Big) dt \,, \label{SDEq+}
\end{align}
where $W$ is a standard Brownian motion obtained by rescaling the initial Brownian motion $b$ (more precisely $W(t) =-\sqrt{4\beta} \, b(t/(4\beta))$ during the $-$ phase and $W(t) =\sqrt{4\beta} \, b(t/(4\beta))$ during the $+$ phase). The diffusions $q_\mu^\pm$ may explode to $-\infty$ in a finite time. By definition, the diffusion $q^{\beta}_\mu$ alternates between $q^+_\mu$ and $q^-_\mu$: it starts to follow $q_\mu^+$ and each time $q^{\beta}_\mu =q_\mu^+$ (resp. $q_\mu^-$) reaches $-\infty$, $q^{\beta}_\mu$ immediately restarts from $+\infty$ and follow $q_\mu^-$ (resp. $q_\mu^+$) .

Let us define the critical line:
\begin{align}\label{criticalline}
	c_{\mu}(t) := \mu + t/4\,,\quad t\geq 0\,.
\end{align}

Roughly, the diffusion $q^-_\mu$ behaves as follows for small values of $\beta$ and after time $0$ or an explosion time of $q^+_\mu$: It starts at $-\infty$ and first quickly goes up to $0$ due to the strong drift term $\exp(-q^-_\mu(t)/\beta)$. Then it spends some time between $0$ and the affine line $c_{\mu}(t)$ where it behaves approximately as a reflected Brownian motion with drift $-a/4$. If it reaches the line $t \mapsto c_{\mu}(t)$ in a finite time then it quickly explodes to $+\infty$ after this hitting time due to the strong drift term $\exp(-( t/4 + \mu)/\beta -q^-_\mu(t) )$ and it immediately restarts from $-\infty$, switching to $q^+_\mu$.

The behavior of the diffusion $q^+_\mu$ is similar except that, it behaves approximately as a reflected Brownian motion with drift $(a+1)/4$ in the region $\{(t,x),\; 0 \leq t,\; 0 \leq x \leq c_\mu(t)\}$. Therefore, it almost surely hits $c_{\mu}(t)$ when $a \geq 0$.

There are two types of explosions for $q^{\beta}_\mu$. Either $q^{\beta}_\mu$ explodes at a time $\xi$ such that $q^{\beta}_\mu(\xi^-) = q^-_\mu(\xi^-)$, which corresponds to the (rescaled) hitting times of $0$ by the initial diffusion $p^{\beta}_{\gl_\beta}$. Alternatively, $q^{\beta}_\mu$ explodes at time $\xi$ such that $q^{\beta}_\mu(\xi^-) = q^+_\mu(\xi^-)$, which corresponds to the (rescaled) explosion times of the initial diffusion $p^{\beta}_{\gl_\beta}$.

In the following, we denote by $\xi^-_{\mu,\beta}(1) < \xi^+_{\mu,\beta}(1) < \xi^-_{\mu,\beta}(2) < \xi^+_{\mu,\beta}(2) < \cdots$ the explosion times of the diffusion $q^\beta_\mu$ and by
\begin{align}\label{explosionsq-}
	\nu^\beta_\mu := \sum_{i \geq 1} \delta_{\xi^+_{\mu,\beta}(i)}\,,
\end{align}
the measure corresponding to the (rescaled) explosions of $p^\beta_{\gl_\beta}$.

From the discussion above, it is natural to expect that as $\beta \to 0$, the trajectory of the diffusion $q^{\beta}_\mu$ converges in law (under a well-chosen topology that smooths out the explosions parts) towards the diffusion $r_\mu$  described in the introduction. Furthermore, notice that the diffusions are all driven by the same underlying Brownian motion $W$ which rigorously motivates our choice of coupling for $r_\mu$. 

\subsubsection{Convergence towards the limiting measures.}

We can now state the desired convergence result:

\begin{proposition}[Convergence of the explosion times]\label{propo:CVexplosions}
	When $\beta \to 0$, the measure $\nu^\beta_\mu$ converges in law to the measure $\nu_\mu$ under the topology of weak convergence. 
\end{proposition}

From our construction, it is straightforward to extend this proposition to the joint law of $$(\nu^\beta_{\mu_1}, \cdots, \nu^\beta_{\mu_k})$$ for fixed positive numbers $\mu_1<\dots<\mu_k$, since the limiting candidates are defined through diffusions driven by the same underlying Brownian motion. This directly implies convergence for the finite-dimensional distributions of the point process $\{\mu_{\beta,a}(i),\; i\geq 1\}$. Let $\mathcal{M}_{\beta,a}  := \sum_{i\geq 1} \delta_{\mu_{\beta,a}(i)}$ denote the measure associated to this point process. Recalling that $\mathcal{M}_{\beta,a}[\mu,+\infty) = \nu^\beta_{\mu}(\R_+)$ and $\mathcal{M}_{a}[\mu,+\infty) = \nu_{\mu}(\R_+)$ for any $\mu >0$, we obtain the following result.
\begin{proposition}[Convergence of finite-dimensional distribution]
	Fix $0 < \mu_1 < \cdots < \mu_k$. 
	When $\beta \to 0$, the random vector 
	\begin{align*}
		\big(\mathcal{M}_{\beta,a}[\mu_1,+\infty), \cdots,  \mathcal{M}_{\beta,a} [\mu_k,+\infty)\big)\,,
	\end{align*}
	converges to 
	\begin{align*}
			\big(\mathcal{M}_{a} [\mu_1,+\infty), \cdots,  \mathcal{M}_{a} [\mu_k,+\infty)\big)\,.
	\end{align*}

\end{proposition}

Theorems \ref{theoCV} and \ref{theoCharact} are a direct consequence of this proposition. We consider the point process as a random measure on the locally compact space $(0,+\infty]$. We endow the space of Radon measures with the vague topology. Note that this is equivalent to vague convergence near $0$ and weak convergence near $+\infty$. More precisely, it corresponds to the topology that makes continuous the maps $m \mapsto \langle f, m \rangle$ for any continuous function $f \; :\; (0,+\infty] \to \R$ supported in some $[\delta,+\infty]$.

Using Kallenberg's tightness condition (see e.g. \cite{Kallenberg}, Lemma 14.15), the family $(\mathcal{M}_{\beta,a})_\beta$ is tight if for all $\delta >0$ and $\eps >0$, there exists $c >0$ such that 
\begin{align*}
	\sup_{0<\beta \leq \beta_0} \P\Big[\mathcal{M}_{\beta,a}([\delta,+\infty)) > c\Big] < \eps\,.
\end{align*}

As $\mathcal{M}_{\beta,a}([\delta,+\infty))$ converges in distribution towards $\mathcal{M}_{a} [\delta,+\infty)$, which is almost surely finite as observed above (we have seen that it is stochastically dominated by a geometric random variable), the condition above is satisfied.

Since the finite-dimensional distributions of any subsequential limit are uniquely identified, we deduce the convergence of the eigenvalue point process stated in Theorem \ref{theoCV} as well as the identification of the limit provided in Theorem \ref{theoCharact}.

\medskip

We conclude this section by outlining the remainder of the paper. Section \ref{CVexplosiontimes} is devoted to the proof of Proposition \ref{propo:CVexplosions}. In subsection \ref{subsec:explosions}, we will first control the first explosion time of the diffusion $q_\mu^\beta$  and deduce the weak convergence of the first $k$ explosions times. Then, in subsection \ref{subsec:tightexplo}, we establish the tightness of the family $(\nu^\beta_\mu)_{\beta >0}$. Finally, in the last section \ref{sec:aympkexplo}, we provide the proof of Proposition \ref{propo:asympkexplo}.

\section{Convergence of the explosion times}\label{CVexplosiontimes}

The goal of this section is to provide a proof of the convergence result stated in Proposition \ref{propo:CVexplosions}. We will first control the explosion times until some large fixed time $T$. 

\subsection{Convergence of the explosion times until a fixed time $T$}\label{subsec:explosions}

In this subsection, we fix $\mu >0$. Let us first consider the first explosion time $\xi_1 := \xi_{\mu,\beta}^-(1)$ of the diffusion $q^\beta = q_\mu^\beta$. By definition $q^\beta(0) = -\infty$ and $q^\beta(t) = q^-(t)$ follows the SDE \eqref{SDEq-} until this first explosion time. 

We define its first hitting time of $0$ by $\tau_0$ and its first hitting time of the critical line $c_\mu$ by $\tau_\mu$.

We decompose the trajectory of $q^-$ into three phases: 
\begin{itemize}
	\item[(A)] \textbf{Ascent from $-\infty$.} First it reaches the axis $x= 0$ in a short time. 
	\item[(B)] \textbf{Diffusion.} Then it spends some time of order $O(1)$ in the region $[0,c_\mu(t)]$, behaving essentially like a reflected Brownian motion (above $0$) with a constant drift $-a/4$. 
	\item[(C)] \textbf{Explosion to $+\infty$.} If it reaches the critical line $t \mapsto c_\mu(t)$ then it will explode with high probability within a short time.
\end{itemize}

In the next lemma, we examine the phases $(A)$ and $(C)$. 

\begin{lemma}[Ascent from $-\infty$ and explosion to $+\infty$]\label{lem:ascentexplo}
With probability going to $1$ as $\beta \to 0$, we have $\tau_0 \leq \beta$ and on the event $\{\tau_\mu < +\infty\}$, $\xi_1 - \tau_\mu \leq \beta$.
\end{lemma}
\begin{proof}
Let us first examine the \emph{ascent from $-\infty$} of the diffusion $q^-$.
Notice that $q^-$ is stochastically bounded from below (until its first explosion time) by the solution $\tilde q^-$ of the SDE:
\begin{align*}
	d \tilde q^-(t) = dW(t) + \frac{1}{4} \big(- a + \exp(-\tilde{q}^-(t)/\beta)\big) dt\,,\quad \tilde q^-(0) = -\infty\,.
\end{align*}

Define $\hat{q}^-(t) := \tilde q^-(t) - W(t) +a t /4$. It solves the random ODE:
\begin{align*}
	d\hat{q}^-(t) = \frac{1}{4} \exp(-\hat{q}^-_\mu(t)/\beta) \exp(-(W(t) - at/4)/\beta) dt
\end{align*}

For any $c >0$, let $g$ denote the solution of the deterministic ODE $g'(t) = c\, \exp(-g(t)/\beta)$ with initial condition $g(0) = -\infty$. It is given by $g(t) = \beta \ln(\frac{c}{\beta} t)$, which is increasing with respect to the parameter $c$.

Define the event 
$\mathcal{E}_0 := \{\forall t \in [0,\beta^5],\;|W(t)| \leq \beta^2\}$, whose probability goes to $1$ as $\beta \to 0$. 
On this event, we have the following lower bound:
\begin{align*}
	\exp(-(W(t) + at/4)/\beta)/4 \geq \exp(- (\beta^2 + |a| \beta^5/4)/\beta)/4 > 1/8\,,
\end{align*}
when $\beta$ is small enough.

Therefore, on $\mathcal{E}_0$, for sufficiently small $\beta$, the diffusion $\tilde q^-$ rises faster than the deterministic solution with parameter $c=1/8$, reaching the level $5 \beta \ln \beta$ before time $\beta^5$. 
After reaching $5 \beta \ln \beta$, the exponential drift remains positive, so $\tilde q^-$ is bounded from below by a Brownian motion with drift $-a/4$ starting at $5 \beta \ln \beta$. This Brownian motion reaches $0$ with probability going to $1$ within any time interval $\gg \beta^2 (\ln \beta)^2$ (note that the drift plays a negligible role here since we are dealing with tiny intervals). We conclude that $\tau_0 \leq \beta$.

\medskip

For the \emph{explosion after $\tau_\mu$}, the proof follows the same lines. We first compare the diffusion with a Brownian motion with drift to show it reaches $c_\mu(\tau_\mu) + 5 \beta \ln \beta$ (starting from time $\tau_\mu$) within a time smaller than $\beta/2$. Then we bound from below our diffusion by 
\begin{align*}
d\tilde q^-(t) &= dW(t) + \frac{1}{4} \big(- a + \exp(-( t/4 + \mu -\tilde q^-(t))/\beta  ) \big) dt\,.
\end{align*}
Comparing again this diffusion with the deterministic solution of the ODE 
\begin{align*}
	g'(t) = c \exp(- (c_\mu(\xi_\mu)-g(t))/\beta)\,,
\end{align*}
with $g(0) = c_\mu(\tau_\mu) + 5 \beta \ln \beta$. By working on the high probability event where the Brownian motion $(W(t) - W(\tau_\mu),\; t \geq \tau_\mu)$ does not reach high values, we can pick some well chosen constant $c$ such that this deterministic lower bound explodes within a time $\beta^5$, concluding the proof.
\end{proof}

We now focus on the region between the line $x=0$ and the critical line $t \mapsto c_\mu(t)$ i.e. phase $(B)$. To simplify the proof, we couple the diffusions under consideration. Here the coupling is quite simple: the Brownian motion driving $q^-$ in \eqref{SDEq-} and the one driving its limiting counterpart $r^-$ have to be equal.

\begin{proposition}[Comparison with a reflected Brownian motion]\label{propo:compdiffreflectedBM}
	Fix $T >1$ (independent of $\beta$). With probability going to $1$ when $\beta \to 0$, we have that for all $t \in [\tau_0, \tau_\mu \wedge T]$, 
	\begin{align*}
		|q^-(t) - r^-(t)| \leq \beta^{1/8}\,.
	\end{align*}
\end{proposition}

\begin{proof}

Note that $\tau_0 < \beta$ with probability going to $1$ thanks to the previous lemma; we work on this event in the following.

	Let $\delta_1 := \beta^{1/4}$.  We first consider the diffusions up to time $\tau_{\mu - \delta_1}$. We first define two auxiliary diffusions $r_1$ and $r_2$ such that, on an event with probability tending to 1, as $\beta \to 0$, the processes $q^-$ and $r^-$ are squeezed between them:
	\begin{align}
\forall t \in [\tau_0, \tau_{\mu-\delta_1} \wedge T],\quad	r_2(t) \leq	q^-(t) \leq r_1(t),\qquad r_2(t) \leq r^-(t) \leq r_1(t)\,.\label{orderingr1r2}
	\end{align}
	
	Let us start with the \emph{upper bound}.
	Define $\eps_1 :=  \exp(- \delta_1/\beta)/2$, and let $r_1$ be a Brownian motion with drift $-a/4 + \eps_1$, reflected at $\delta_1$, with starting point $r_1(\tau_0) := \delta_1$ driven by the same Brownian motion as $r^-$ and $q^-$. 
	
	In the time interval $[\tau_0, \tau_{\mu - \delta_1}]$, $q^-$ is bounded from above by $r_1$.
	Observe indeed that the drift induced by the two exponential terms in \eqref{SDEq-} in the region $\{(t,x),\; t \geq \tau_0,\; x \in [\delta_1, c_{\mu-\delta_1}(t)]\}$ is bounded from above by $\eps_1$.
	
Moreover, we also have $r^-(t) \leq r_1(t)$ for all $t \geq \tau_0$ on the event $\{\tau_0 \leq \beta,\; r^-(\tau_0) \leq \delta_1\}$ (which has probability tending to $1$).
	
	\medskip
	
	For the \emph{lower bound}, let us define $r_2$ starting from $-\delta_1$ at time $\tau_0$, following a Brownian motion with drift $-a/4$ reflected at $-\delta_1$ driven by the same Brownian motion. As long as $q^-$ does not reach $-\delta_1$, the diffusion $r_2$ is a lower bound for $q^-$. 
	
	We must verify that the probability of $q^-$ reaching $-\delta_1$ before $\tau_{\mu}$ tends to $0$. This is guaranteed by the strong positive drift of $q^-$ strictly below $0$. More precisely, consider the stationary diffusion starting at $0$, satisfying
	\begin{align*}
		d \tilde q(t) = (-a/4 + \exp(-\tilde q(t)/\beta)/4) dt + dW(t)\,.
	\end{align*}
Since $\tilde q$ has a smaller drift than $q^-$, it is more likely to hit $-\delta_1$. We compute the probability that $\tilde q$ hits $-\delta_1$ before $\mu$ using the scale function.
	\begin{align*}
	\P[\tilde q \; \mbox{ hits }- \delta_1 \mbox{ before }\mu] &=	\frac{\int_0^\mu \exp(- 2  \int_0^y (-a/4 + e^{-x/\beta}/4) dx)dy}{\int_{-\delta_1}^{\mu} \exp(- 2 \int_0^y (-a/4 + e^{-x/\beta}/4) dx) dy} \,.
	\end{align*}
	Splitting the integral in the denominator into its integral from $-\delta_1$ to $0$ and from $0$ to $\mu$, one can see that the numerator term is of order $O(1)$ whereas the denominator grows exponentially as $\beta \to 0$. Thus this probability tends to $0$.

	It is also immediate to check that $r^-$ is bounded below by $r_2$ as long as $r^-(\tau_0) \geq -\delta_1$, which happens with probability going to $1$.
	
	\medskip
	
Gathering the bounds, with probability going to $1$, we have the inequalities \eqref{orderingr1r2}.
	
	\medskip
	
Finally, we bound the difference between $r_1$ and $r_2$. Define the shifted space/time versions $\tilde r_1(t) := r_1(t + \tau_0) - \delta_1$ and $\tilde r_2(t) := r_2(t + \tau_0) + \delta_1$ of $r_1$ and $r_2$ so that both are now reflected above $0$ and start at $0$ at time $t=0$. Let $s(t) :=  \tilde r_1(t) - \tilde r_2(t) \geq 0$ be their difference. We have:
	\begin{align*}
		d s(t) = \eps_1 dt + dL_1(t) - dL_2(t)
	\end{align*}
	where $L_1$ and $L_2$ are the local time at $0$ of $\tilde r_1$ and $\tilde r_2$ respectively.
	It gives:
	\begin{align*}
		d(s^2(t)) = 2  s(t) (\eps_1 dt + dL_1(t) - dL_2(t))
	\end{align*}
	
Note that $L_1$ increases only when $\tilde r_1 (t) = 0$, implying $s(t) = \tilde r_2(t) \geq 0$, thus $2 s(t) dL_1(t) \leq 0$. Similarly, $-s(t) dL_2(t) \leq 0$. Therefore:
\begin{align}
s(t)^2 \leq 2 \eps_1 \int_0^t s(u) du \,. \label{ineqs2}
\end{align}

If we define $S(t) := \int_0^t s(u) du$. Since $s(t) = S'(t) \geq 0$, we have $(S'(t))^2 \leq 2 \eps_1 S(t)$, which implies $S'(t) \leq \sqrt{2 \eps_1} \sqrt{S(t)}$. Using Bihari LaSalle inequality\footnote{One could avoid using this inequality by introducing the time $t_0 := \inf\{u \geq 0, \; S(u) >0\}$ which leads to $\frac{d}{dt} \sqrt{S(t)} \leq  \sqrt{2 \eps_1}/2$ for $t \geq t_0$.}, we obtain for $t \geq 0$,
\begin{align*}
	s(t) \leq \eps_1 t\,.
\end{align*}
Returning to the original processes for $t \geq \tau_0$, we have 
\begin{align*}
	|r_1(t) - r_2(t)| \leq |\tilde r_1(t-\tau_0) - \tilde r_2(t-\tau_0) + 2 \delta_1| \leq \eps_1 T + 2 \delta_1\,.
\end{align*}
Since $\delta_1 = \beta^{1/4}$ and $\eps_1 =  \exp(-\beta^{-3/4})/2$, this quantity is bounded by $2\beta^{1/4}$ for small enough $\beta$. Combined with the ordering \eqref{orderingr1r2}, this concludes the proof up to the time $T \wedge \tau_{\mu - \delta_1}$.
	
	It remains to extend this bound up to the time $\tau_{\mu}$. Suppose that $\tau_{\mu - \delta_1} < T$. By comparing $q^-$ with a Brownian motion with constant drift $-a/4$, we see that $q^-$ hits the affine line $c_{\mu}$ within an additional time of order $\beta^{3/8} \gg \delta_1^2$ with probability going to $1$ (note that the drift is irrelevant over such a short interval). Over a time interval $\beta^{3/8}$, the process $r^-$ fluctuates less than $\beta^{1/8}/2 \gg \beta^{3/16}$ with overwhelming probability. This yields the desired bound up to time $\tau_\mu \wedge T$.
\end{proof}

Note that the exact same proofs as Lemma \ref{lem:ascentexplo} and Proposition \ref{propo:compdiffreflectedBM} can be applied for the regime $q^\beta = q^+$ as we did not use the specific value of the drift $-a/4$. In particular, we deduce the following proposition:
\begin{proposition}\label{propo:controlexplokT}
	Fix $k \geq 1$ and $\mu >0$. With probability going to $1$, we have: 
	\begin{align*}
	\forall i \in \{1,\cdots,k\},\quad |	\xi^{\pm}_{\mu,\beta}(i) \wedge T -\xi^{\pm}_\mu(i) \wedge T | \leq \beta^{1/8^{2k}}\,.
	\end{align*}
\end{proposition}
\begin{proof}
	Note that our control on the trajectories of Proposition \ref{propo:compdiffreflectedBM} is a priori not sufficient to control hitting times as two trajectories can be very close in uniform norm but have well-separated hitting times of an affine line. 	
	It nevertheless gives that if $\tau_\mu \leq T$, $|q^{-}(\tau_\mu) - r^-(\tau_\mu)| \leq \beta^{1/8}$. Therefore, $r^-$ should hit $\mu - \beta^{1/8}$ before $\tau_\mu$. Moreover, if again this hitting time $\xi^{-}_{\mu - \beta^{1/8}}(1)$ is smaller than $T \wedge \tau_\mu$, we have that $q^-$ should hit $\mu - 2 \beta^{1/8}$ before $\xi^{-}_{\mu}(1)$. This gives:
	\begin{align*}
		\tau_{\mu - 2 \beta^{1/8}}  \wedge T \leq  \xi^-_{\mu - \beta^{1/8}}(1)  \wedge T \leq \tau_{\mu} \wedge T\,.
	\end{align*}
But the probability that $\tau_{\mu} - 	\tau_{\mu - 2 \beta^{1/8}} \leq \beta^{1/8}/4$ tends to $1$ when $\tau_{\mu - 2 \beta^{1/8}} < \infty$ as we can bound from below the process $q^-$ again by a Brownian motion with constant drift starting at $c_{\mu - 2 \beta^{1/8}}(\tau_{\mu - 2 \beta^{1/8}} )$ which has an overwelming probability to hit $c_\mu$ within short time (notice again that the drifts are irrelevant here as we work on a small time-interval). And similarly, for the diffusion $r^-$, we obtain $\xi^-_{\mu}(1)-\xi^-_{\mu - \beta^{1/8}}(1) \leq \beta^{1/8}/4$ when $\xi^-_{\mu - \beta^{1/8}}(1)  < \infty$.

We obtain that with probability tending to $1$, we have
\begin{align}
	|	\xi^{-}_{\mu,\beta}(1) \wedge T -\xi^{-}_\mu(1) \wedge T | \leq \beta^{1/8}\,\label{exploclose}
\end{align}

For the next explosion, we work on the event that the ascent of $q^\beta = q^+$ from $-\infty$ (right after the explosion time $\xi^{-}_{\mu,\beta}(1)$) is smaller than $\beta$. This event occurs with overwelming probability thanks to Lemma \ref{lem:ascentexplo}. We also work on the event \eqref{exploclose}.  We have now two diffusions $q^+$ and $r^+$ starting at $0$, but at two different starting times $\tau'_0$ and $\tau''_0$ separated by at most $\beta^{1/8}/2 + \beta$. It suffices to apply the proof of the previous proposition, using as an imput $\beta^{1/8}$ instead of $\beta$. We obtain that 
\begin{align*}
	|q^+(t) - r^+(t)| \leq \beta^{1/64}\,.
\end{align*}
for $t \in [(\tau'_0 \vee \tau''_0)\wedge T,  \tau'_\mu\wedge T]$ where $\tau'_\mu$ is the first hitting time of the critical time by $q^\beta = q^+$ after time $\xi^{-}_{\mu,\beta}(1)$.
Iterating the proof gives the result.
\end{proof}

This proposition in particular implies that the explosion times of $q^\beta$ are ``macroscopic'' (of order $O(1)$) as $\beta \to 0$ with overwelming probability.
Indeed, it takes a strictly positive time (independent of $\beta$) for $r^\pm$ to reach the critical line $c_\mu$ and the explosion times of $q^\beta$ are all close to these limiting hitting times with high probability. Consequently, the number of explosions occuring before any fixed time $T$ is bounded almost surely. 
This shows that Proposition \ref{propo:controlexplokT} controls \emph{all} explosion times of $q^\beta$ up to time $T$ with probability tending to $1$, rather than merely the first $k$ ones.

\subsection{Tightness of the explosion times of the diffusion $q$}\label{subsec:tightexplo}

%
%
%
We would like to prove that for all $\eps >0$, there exists $T >0$ such that
\begin{align*}
 \inf_{\beta \leq \beta_0}\P[\nu_\beta[T,\infty) =0] \geq 1- \eps\,.
\end{align*}

Recall the dynamics of $q^\beta_\mu$ given by \eqref{SDEq-} and \eqref{SDEq+}. In the following, we drop the subscript $\mu$ to simplify the notations. We would like to prove the following lemma:
\begin{lemma}[No explosions after large times]\label{tightnessexplosions}
For any $\eps >0$, there exists a time $T >0$ and $\beta_0 >0$, such that after time $T$, for all $\beta < \beta_0$, there is no explosion of the diffusion $q^\beta$ with probability greater than $1- \eps$.
\end{lemma}

The main difficulty for this part is that we do not know where the diffusion is at time $T$: in a worst-case scenario, it could be close to the critical line, which would then lead to an explosion after time $T$.

\begin{proof} Let us first suppose that $a \geq 0$ and take some small $\beta$ (the argument below will work for all $\beta$ smaller than some fixed deterministic $\beta_0$).
	
	We analyze the trajectory starting from time $T$ and show that, on an event of probability tending to $1$ when $T$ goes to infinity while at most two explosions may occur in the interval $[T,T^9]$, there will be no explosion after time $T^9$.
	
	First, note that if $q^\beta$ is in the $q^+$ phase at time $T$, its drift is bounded from below by $(a+1)/4 \geq 1/4$. Therefore, it will hit the critical line almost surely. The time it takes to explode is $O(T)$ if $a >0$ and $O(T^2)$ if $a=0$. In the worst-case scenario ($a=0$ and $q^+(T) = -\infty$), it explodes before time $T^3$ with high probability (for some large $T$), since a Brownian motion with drift $1/4$ hits the critical line $\mu+t/4$ before $T^3$ with high probability. Therefore, the process is in the $q^-$ phase at some time prior to $T^3$. Using the strong Markov property, we can restrict our analysis to the case where the process is in the $q^-$ phase at time $T$, but we must now prove that there is no explosion after time $T^3$, to ensure there is no explosion after time $T^9$ for our original process.
	
	Fix $\delta \in (0,1/8)$.
	We introduce two auxiliary diffusions $r_1$ and $r_2$, driven by the same Brownian motion as $q^-$. The diffusion $r_1$ starts at time $T$ from position $1$ and is a reflected Brownian motion above $1$ with drift $-a/4 + \delta$. The diffusion $r_2$ starts at time $2T$ from $c_\mu(T)+3 T/16$ and follows a reflected Brownian motion with drift $-a/4 + \delta$ above $1$.

	We divide the analysis according to the position of $q^\beta$ at time $T$.
	\begin{itemize}
		\item \emph{If $q^\beta(T) \leq 1$.} If $\beta$ is small enough, the drift of the diffusion $q^\beta$ is bounded from above by the drift of $r_1$ for all values in $[1,c_\mu(t)]$ (as $q^\beta$ is in its $q^-$ phase). Therefore $q^\beta$ stays below $r_1$ as long as $r_1$ does not hit the critical line.
		\item  \emph{If $q^\beta(T) \geq c_\mu(T)$}. Then with high probability, $q^\beta$ explodes rapidly thanks to Lemma \ref{lem:ascentexplo}. After this explosion time, $q^\beta$ is in its $q^+$ phase, which explodes with high probability before time $T^3/2$ using the same arguments as above. Then it is back to its $q^-$ phase starting at $-\infty$. In particular, it stays below $r_1$ as long as $r_1$ does not hit the critical line.
		\item \emph{If $1 < q^\beta(T) < c_\mu(T)$}. Let $\tau_T := \inf\{t \geq T,\; q^\beta(t) = c_\mu(t)\}$ be the next hitting time of the critical line. 
		\begin{itemize}
			\item 	If $\tau_T < 2T$, then the diffusion $q^\beta$ explodes quickly after $\tau_T$ thanks to Lemma \ref{lem:ascentexplo}. It then enters the $q^+$ phase, explodes again (before time $T^3/2$) and re-enters the $q^-$ phase at $-\infty$. It will therefore stay below $r_1$ as long as $r_1$ does not hit the critical line.
			\item 	If $\tau_T \geq 2T$, then during the time interval $[T, 2T]$, $q^\beta$ is bounded from above by a reflected Brownian motion with drift $-a/4 + \delta$ starting at time $T$ at position $c_\mu(T)$: with probability tending to $1$ when $T \gg 1$, it will be below $c_\mu(T) +3T/16$ at time $2T$. (Indeed, the probability that a Brownian motion with drift smaller than $1/8$ ends above its starting point plus $3T/16$ decays exponentially with $T$). After time $2T$, on this high probability event, the diffusion $q^\beta$ is therefore bounded by $r_2$ as long as $r_2$ does not hit the critical line.
		\end{itemize}	
	\end{itemize}

In all cases, before time $T^9$, the process $q^-$ becomes trapped below either $r_1$ or $r_2$. 

Those two diffusions have a drift smaller than $1/8$. Since this is strictly less than the slope of the critical line, the probability that they hit the critical line is exponentially decaying with $T$ (see e.g. Proposition \ref{propo:global_hitting_prob} which is a more precise statement that implies this fact).

Therefore whatever its phase or position at time $T$, the diffusion $q^\beta$ will never explode after time $T^9$ on this high probability event.

	The case $a <0$ follows the exact same lines, except it is simpler as we no longer have a phase with a probability of explosion equal to $1$. For either phase, we can bound our diffusion with reflected Brownian motions with drift in a similar way than above.
\end{proof}

\subsection{Proof of Proposition \ref{propo:CVexplosions}}

We now have all the ingredients to conclude the proof of Proposition \ref{propo:CVexplosions}, namely the convergence in law of the random measures $\nu^\beta_\mu \to \nu_\mu$ for the weak topology. We will prove in fact the stronger convergence in probability thanks to our coupling. The argument is very standard but we chose to include it here for completeness.

Recall that $\nu^\beta_\mu = \sum_{i \ge 1} \delta_{\xi^+_{\mu,\beta}(i)}$ and $\nu_\mu = \sum_{i \ge 1} \delta_{\xi^+_\mu(i)}$. As the space of finite measures on $\R_+$ equipped with the weak topology is separable and metrizable, it suffices to show that for any bounded continuous function $f \colon \R_+ \to \R$ and any $\delta > 0$, 
\begin{align*}
	\lim_{\beta \to 0} \P\Big[ \big| \langle \nu^\beta_\mu, f \rangle - \langle \nu_\mu, f \rangle \big| > \delta \Big] = 0\,.
\end{align*}

Fix such a function $f$ and let $\eps > 0$. As the total mass $\nu_\mu(\R_+)$ is finite almost surely and its atoms are finite, we can choose a large integer $K$ and a large time $T$ such that
\begin{align}\label{eq:bound_limit_measure}
	\P\big[\nu_\mu(\R_+) \le K \quad \text{and} \quad \xi^+_\mu(K) \le T \big] \ge 1 - \eps/3\,.
\end{align}
By the tightness of the explosion times (Lemma \ref{tightnessexplosions}), we may increase $T$ and $K$ if necessary so that for all $\beta$ sufficiently small,
\begin{align}\label{eq:bound_beta_measure}
	\P\big[\nu^\beta_\mu([T, \infty)) = 0 \quad \text{and} \quad \nu^\beta_\mu(\R_+) \le K \big] \ge 1 - \eps/3\,.
\end{align}
Let $\mathcal{A}_{\beta}$ denote the intersection of these two events. On $\mathcal{A}_{\beta}$, both point processes have at most $K$ points, all strictly confined to $[0, T]$, which implies $N_\beta := \nu^\beta_\mu(\R_+) = \nu_\mu(\R_+) =: N \le K$.

By Proposition \ref{propo:controlexplokT}, the first $K$ hitting times are pathwise close. That is, there exists $\beta_0 > 0$ such that for all $\beta < \beta_0$, the event
\begin{align*}
	\mathcal{B}_\beta := \Big\{ \max_{1 \le i \le K} \big|\xi^+_{\mu,\beta}(i) \wedge T - \xi^+_\mu(i) \wedge T\big| \le \beta^{1/8^K} \Big\}
\end{align*}
has probability at least $1 - \eps/3$. Let $\omega_{f, T}$ denote the modulus of continuity of $f$ on $[0, T]$. On the event $\mathcal{A}_{\beta} \cap \mathcal{B}_\beta$, we have
\begin{align*}
	\big| \langle \nu^\beta_\mu, f \rangle - \langle \nu_\mu, f \rangle \big| 
	&\le \sum_{i=1}^{N} \big| f(\xi^+_{\mu,\beta}(i)) - f(\xi^+_\mu(i)) \big| \\
	&\le K \cdot \omega_{f, T}\big(\beta^{1/8^K}\big)\,.
\end{align*}

Since $\omega_{f, T}(\beta^{1/8^K}) \to 0$ deterministically as $\beta \to 0$, this upper bound is strictly less than $\delta$ for $\beta$ sufficiently small. As $\P[(\mathcal{A}_{\beta} \cap \mathcal{B}_\beta)^c] \le \eps$, the desired convergence in probability follows. \qed

\section{Asymptotics for the $k$ largest eigenvalues}\label{sec:aympkexplo}
The goal of this section is to provide the proof of Proposition \ref{propo:asympkexplo}. Let $R^{(b)}$ denote the reflected Brownian motion with drift $b$. Recall that it is given by
\begin{align*}
	R^{(b)}(t) = 	Y^{(b)}(t) + \sup_{0 \leq u \leq t} \max(0,-Y^{(b)}(u))
\end{align*}
where $Y^{(b)}(t) = b t + B(t)$ is a (non-reflected) Brownian motion with drift $b$. In our procedure, we deal with reflected Brownian motions whose drift alternates between $-a/4$ (the $-$ phase) and $(a+1)/4$ (the $+$ phase). We define the respective drift parameters as:
\begin{align}
	a_1 := -a/4, \qquad a_2 := (a+1)/4\,.\label{defai}
\end{align}

We are interested in the event where there exist at least $k$ hitting times of $c_\mu$ by the $+$ phase, which corresponds to $2k$ alternating phases hitting the line. When $a \geq 0$, since the final $+$ phase hits the critical line almost surely, the event is a.s. equal to the one where we have at least $2k-1$ hitting times. When $a \in (-1,0)$, both events are rare events, and we need indeed $2k$ hitting times.

Let $\tau^{(1)}_\mu$ and $\tau_\mu^{(2)}$ denote the first hitting times of $c_\mu$ by $R^{(a_1)}$ and $R^{(a_2)}$ respectively. Note that typically, the diffusions $R^{(a_1)}$ and $R^{(a_2)}$ when $a \in (-1,0)$ do not reach the critical line $c_\mu$ in a finite time when $\mu \gg 1$. Conversely, when $a \geq 0$, $R^{(a_2)}$ reaches it in a finite time almost surely in a typical time of order $4 \mu/a$ (for $a >0$). 

\subsection{Large deviation events}
We start by establishing the asymptotic probability that the $-$ phase hits the critical line in a finite time which implies Proposition \ref{propo:asympkexplo} for $k=1$ and $a \geq 0$.
\begin{proposition}\label{propo:global_hitting_prob}
	Let $a > -1$. There exist a constant $C_0(a) >0$ such that for all $\mu > 0$,
	\begin{align*}
		C_0^{-1}(a) \exp\Big(-\mu \; \Big(\frac{1+a}{2}\Big)\Big) \leq \P\big(\tau^{(1)}_\mu < \infty\big) \leq C_0(a) \exp\Big(-\mu \; \Big(\frac{1+a}{2}\Big)\Big)\,.
	\end{align*}
\end{proposition}

\begin{proof}
	We apply the Girsanov formula to change the drift of the reflected Brownian motion from $a_1 = -a/4$ to a constant drift $b > 1/4$ so that the line is hit almost surely. Denote by $\tau_b$ the first hitting time of $c_\mu$ by $R^{(b)}$. We have
	\begin{align*}
		\P\big[\tau^{(1)}_\mu < \infty \big] = \E[\exp(G_{\tau_b}(Y^{(b)}))1_{\{\tau_b< \infty\}}] \,,
	\end{align*} 
 (as the change of drift is constant, Novikov's condition is trivially satisfied on finite intervals: one applies the change of measure up $\tau_b \wedge T$ and passes to the limit as $T \to \infty$). 
At the hitting time $\tau_b$, the exponent of the Radon-Nikodym derivative is:
	\begin{align}
		G_{\tau_b}(Y^{(b)}) &= (a_1 - b) (\mu + \tau/4 - L^{(b)}(\tau_b)) - \frac{1}{2}\tau_b (a_1^2 - b^2) \notag \\
		&= (a_1 - b)\mu + (b - a_1)L^{(b)}(\tau_b) + \tau_b \Big( \frac{a_1 - b}{4} - \frac{a_1^2 - b^2}{2} \Big)\,. \label{expressionG}
	\end{align}
	Let us choose $b$ such that the coefficient in front of the time $\tau_b$ vanishes. It gives $b^\ast := (a+2)/4$. With this choice, the exponent is equal to
	\begin{align*}
		G_{\tau_{b^\ast}}(Y^{(b^\ast)}) = -\mu \frac{a+1}{2} + \frac{a+1}{2} L^{(b^\ast)}(\tau_{b^\ast})\,.
	\end{align*}
One can simply bound the local time by: $0 \le L^{(b^\ast)}(\tau_{b^\ast}) \le L^{(b^\ast)}(\infty)$. The random variable $L^{(b^\ast)}(\infty) = \sup_{t \geq 0} (-B(t) - b^\ast t)$ follows an exponential distribution with parameter $2b^\ast = (a+2)/2$. Therefore, 
\begin{align*}
	\E[\exp(((a+1)/2) L^{(b^\ast)}(\infty))] = \frac{a+2}{2} \int_0^\infty \exp(-x/2) dx \,.
\end{align*}
The expectation $\E[\exp(\frac{a+1}{2} L^{(b^\ast)}(+\infty))]$ is bounded above by $a+2 >0$, which gives the desired upper bound. The lower bound follows directly from the fact that $L^{(b^\ast)} \geq 0$ and $\P[\tau_{b^\ast} < \infty]=1$.
\end{proof}



The exact same proof gives the following asymptotics for $R^{(a_2)}$ when $a \in (-1,0)$.
\begin{proposition}\label{propo:global_hitting_prob2}
	Let $a > -1$. There exist a constants $C_0(a) >0$ such that for all $\mu > 0$,
		\begin{align*}
		C_0^{-1}(a) \exp\Big(-\mu \; \frac{|a|}{2}\Big) \leq \P\big(\tau^{(2)}_\mu < \infty\big) \leq C_0(a) \exp\Big(-\mu \; \frac{|a|}{2}\Big)\,.
	\end{align*}
\end{proposition}

To compute the probability of $n$ successive hitting times, we must optimize the trajectory, and choose appropriate duration for each of the hitting times. 

\begin{proposition}\label{propo:localhitting}
Let $a >-1$ and $i \in \{1,2\}$. Fix $\eps >0$ and $\delta >0$. Recall from \eqref{defai} the definition of $a_i$. Define $t^\ast_i := 1/|a_i-1/4|$ and the rate functions:
%
\begin{align*}
	I_i(t) :=  \frac{ (1 + (1/4-a_i) t)^2}{2t} \,.
\end{align*}
	Then, there exists a constant $C := C(\delta, \eps,a)$ depending only on $a$, $\delta$ and $\eps$ such that,
	for all time $t \in [\eps, t^\ast_i-\eps]$, and all $\mu >0$,
	\begin{align*}
		C^{-1} \mu^{-1/2}\, \exp\big(-\mu \, I_i(t)\big)\leq \P\big(\tau^{(i)}_\mu \in [\mu t ,\mu t+\delta]\big)  \leq C \, \exp\big(-\mu \, I_i(t)\big)\,.
	\end{align*}
Furthermore, for the initial interval $[0,\mu \, \eps]$, there exists $C_0(a)$ depending only on $a$ such that for $i=1,2$, and for all $\mu >0$,
\begin{align*}
\P\big(\tau^{(i)}_\mu \in [0, \mu \, \eps]\big) \leq C_0(a)  \exp\big(-\mu\,  I_i(\eps)\big)\,.
\end{align*}
\end{proposition}


\begin{remark}
	Note that the probabilities are increasing with the time $t$ up to their optimal value at $t^\ast_i = 4/(a+1)$ for $i=1$ and $t^\ast_i=4/|a|$ for $i=2$.
\end{remark}

\begin{proof}
Take $i \in \{1,2\}$. One can again apply Girsanov's formula to change the drift of $Y^{(a_i)}$ to $Y^{(b)}$. Note that by construction, we will always accelerate time meaning we take  $b > a_i$. Denote by $\tau_b$ the first hitting time of $c_\mu$ by $R^{(b)}$. We have
	\begin{align*}
		\P\big[\tau^{(i)}_\mu/\mu \in [t , t+\delta/\mu]\big] = \E[\exp(G_{\tau_b}(Y^{(b)}))1_{\{\tau_b /\mu \in [t , t+\delta/\mu]\}}] \,,
	\end{align*}
	with
	\begin{align*}
		G_T(Y^{(b)}) &= (a_i - b) (Y_T - Y_0) - \frac{1}{2}  T (a_i^2 - b^2)\,.
	\end{align*}
	At time $T = \tau_b$, we get
	\begin{align*}
		G_{\tau_b}(Y^{(b)}) &= (a_i - b) (Y(\tau_b) - Y_0) - \frac{1}{2}  \tau (a_i^2 - b^2)\\
		&= (a_i- b) (\mu + \tau_b/4 - L^{(b)}(\tau_b)) - \frac{1}{2}  \tau (a_i^2 - b^2)\\
		&= -(b -a_i) \mu - \tau_b (b - a_i)/4 + \tau_b (b-a_i )(a_i + b)/2 + (b - a_i) L^{(b)}(\tau_b) \,.
	\end{align*}
	
	Denote 
	\begin{align*}
		I(b,s) := (b-a_i)\big(1 + \frac{s}{4}\big) - \frac{s}{2}(b-a_i)(a_i + b)\,.
	\end{align*}
	We have
	\begin{align*}
		G_{\tau_b}(Y^{(b)}) &= - \mu I(b,\tau_b/\mu) + (b - a_i) L^{(b)}(\tau_b)\,.
	\end{align*}

	We want to minimize the exponent by taking the supremum over $b$ of the quadratic function $I(b,s)$.
	The unconstraint maximum of $I(b,s)$ is reached at $b^\ast = 1/4 + 1/s$ and its value is
	\begin{align*}
		I(b^\ast,s)
		&= \frac{\big( 1 + (\frac{1}{4} -  a_i) s\big)^2 }{2s}
	\end{align*}

	Moreover, we can bound $L^{(b)}(\tau_b)$ by $L^{(b)}(\infty)$. The random variable $L^{(b)}(\infty) = \sup_{t \geq 0} (-B(t) - bt)$ follows an exponential distribution with parameter $2b$. Therefore, 
	\begin{align*}
		\E[\exp((b - a_i) L^{(b)}(\infty))] = (2b) \int_0^\infty \exp(-b x - a_i x) dx \,.
	\end{align*}
	The last integral is finite iff $b + a_i>0$ and equals $2b/(b+a_i)$. When we choose $b=b^\ast$, this always holds for $i=2$. It also holds for $i=1$ since we restricted the times to $t \leq t_1^\ast =4/(a+1) < 4/(a-1)$ for $a >1$ (when $a \in (-1,1)$, it is always satisfied). 
	
	Let us first examine the case $i=2$. Then we obtain the upper bound
	\begin{align*}
		\P\big[\tau^{(2)}_\mu/\mu \in [t , t+\delta/\mu]\big] &\leq 	C \exp\Big(- \mu \big(\frac{1}{2 (t+\delta/\mu)}+ \frac{1}{2} (\frac{a}{4})^2 t - \frac{a}{4}\big) \Big)\\
	\end{align*}
	where $C$ is some constant depending only on $a$. For the lower bound, using that $L^{(b)}(\tau_b) \geq 0$, we obtain
	\begin{align*}
		\P\big[\tau^{(2)}_\mu/\mu \in [t , t+\delta/\mu]\big] \geq 	p_0(\mu) \exp \Big(- \mu \big(\frac{1}{2 t}+ \frac{1}{2} (\frac{a}{4})^2 (t+\delta/\mu) - \frac{a}{4}\big) \Big)\,,
	\end{align*}
	where 
	\begin{align*}
		p_0(\mu) := \P\big[R^{(b^\ast)} \mbox{ reaches } c_\mu \mbox{ in the time interval } [\mu t, \mu t+\delta]\big]\,.
	\end{align*}	
	Recall that the drift is $b^\ast = 1/4 + 1/t$ and $R^{(b^\ast)}(s) = B(s) + b^\ast s + L^{(b^\ast)}(s)$. Let $l > 0$. On the event $L^{(b^\ast)}(\mu t) \leq l$, we have $L^{(b^\ast)}(s) \leq l$ for all $s$, and we obtain:
	\begin{align*}
		p_0(\mu) &\geq \P\Big(L^{(b^\ast)}(\mu t) \leq l, \; \sup_{s \leq \mu t} (B(s)+ s/t) \leq \mu  - l,\; B(\mu t + \delta) \geq - \frac{\delta}{t}\Big)\,.
		\end{align*}
	Using Markov property at time $\mu t$, we obtain the product
	\begin{align}
		p_0(\mu) &\geq \P\Big(L^{(b^\ast)}(\mu t) \leq l, \;  \sup_{s \leq \mu t} (B(s) + s/t) \leq \mu - l,\; B(\mu t) \in [-l-1,-l]\Big) \label{lowerboundp01}\\
		&\qquad \times \inf_{l_0 \in [l,l+1]} \P \Big(B(\delta)  \geq l_0  - \frac{\delta}{t}\Big)\,. \label{lowerboundp02} 
	\end{align}
	The probability \eqref{lowerboundp01} is bounded by the intersection of three events:
	\begin{itemize}
		\item $\{B(\mu t)  \in [ - l -1, - l] \}$. By Gaussian scaling, it is of order $O(1/\sqrt{\mu})$.
		\item Conditioned on the endpoint $B(\mu t) + \mu t/t = x \in [\mu -l-1, \mu-l]$, the process $Y^x(s) := B(s) + s/t$ is a Brownian bridge of length $\mu t$ from $0$ to $x$. We require $\sup_{s \le \mu t} Y^x(s) \le \mu - l$. The probability that a bridge of length $T$ ending at $y$ stays below a barrier $A \ge y$ is exactly $1 - \exp(-2A(A-y)/T)$ (see e.g. Proposition 8.1 in \cite{karatzas}). 
		
		Substituting $A = \mu - l$, $y = x$, and $T = \mu t$, we obtain $A-y = \mu -l-x := \tilde x \in [0, 1]$. The conditional probability is therefore:
		\begin{align*}
			1 - \exp\Big( - \frac{2(\mu - l) \tilde x}{\mu t} \Big) \underset{\mu \to \infty}{\longrightarrow} 1 - \exp\Big( - \frac{2\tilde x}{t} \Big) > 0\,.
		\end{align*}
		Thus, the probability of the path staying below the line is bounded below by a strictly positive constant.
		\item $\{L^{(b^\ast)}(\mu t) \leq l \}$ for the conditioned bridge $Y^x$, which holds if $\inf_{s \le \mu t} (Y^x(s) + s/4) \ge -l$. Adding the linear drift $s/4$ to the Brownian bridge $Y^x(s)$ simply yields a new Brownian bridge of length $\mu t$, starting at $0$ and ending at $y' := x + \mu t / 4$. By symmetry, the probability that its infimum stays above $-l$ is equivalent to the supremum of a bridge ending at $-y'$ staying below $l$. Using the same identity as above from \cite{karatzas}, this probability is exactly $1 - \exp(-2l(l + y')/\mu t)$. Substituting $y'$, the exponent converges asymptotically:
		\begin{align*}
			\frac{2l(l + x + \mu t/4)}{\mu t}\underset{\mu \to \infty}{\longrightarrow} 2l\Big(\frac{1}{t} + \frac{1}{4}\Big) = 2b^\ast l\,.
		\end{align*}
		Thus, this conditional probability converges to $1 - e^{-2b^\ast l}$.
	\end{itemize}
By taking $l$ large enough so that the last probability is sufficiently close to $1$, we obtain a strictly positive probability for the intersection of the three events. Then we get a positive probability for the event \eqref{lowerboundp02} (decreasing with $l$, but this is a fixed parameter), completing the proof.

	%
	%
	%

	Let us now examine the case $i=1$ that is $a_i = -a/4$. As we ask that $t < 4/(a+1) - \eps$, this implies that $1/t + 1/4 > a/4$, therefore we can choose $b = b^\ast$ and we obtain the upper bound:
	\begin{align*}
		\P\big(\tau^{(1)}_\mu/\mu \in [t , t+\delta/\mu]\big) \leq C \exp\Bigg(- \mu \big(\frac{1}{2 (t+\delta/\mu)}+ \frac{1}{2} (\frac{1+a}{4})^2 t + \frac{1+a}{4}\big)\Bigg)\,.
	\end{align*}
	where $C$ is some absolute constant depending only on $a$.
	The lower bound is similar to the case $i=2$.
	
	For the initial interval $[0, \mu \eps]$, the upper bound is obtained by evaluating the supremum of the exponentially small probabilities at $t= \eps$. 
\end{proof}

\subsection{Proof of Proposition \ref{propo:asympkexplo} for $a \geq 0$}
Let us now prove the desired proposition for $a \geq 0$. We discretize the time intervals to bound the probability of multiple hitting times. For clarity, we write it in the case $n=3$ (which corresponds to the event $\{\lambda_2 \geq \mu\}$) although it naturally extends to any $n$.

Let us fix a sufficiently large constant $M > a+2$ such that the probabilities $\P(\tau_\mu^{(i)} \geq \mu M)$ become exponentially negligible in front of the expected decay. Note as well that when $\eps$ is small enough, we also have that $\P(\tau_\mu^{(i)} \leq \mu \eps)$ is exponentially negligible. We then partition $[\mu \eps ,\mu M]$ into a grid of size $\delta = 1$, letting the integers $k_i \in K:= \{\lfloor \eps \mu \rfloor,\cdots,\lceil  M \mu \rceil\}$ for $i=1,2$. We define the successive boundary heights at the end of each phase:
\begin{align*}
	H_1(k_1) &:= \mu + k_1/4\,, \\
	H_2(k_1,k_2) &:= H_1(k_1) + k_2/4\,.
\end{align*}

Using the strong Markov property and the union bound over the elements, we obtain the upper bound
\begin{align*}
	\P(\gl_2 \geq \mu) \leq \big(\mu M\big)^2 \sup_{k_1,k_2 \in K} \Big[ \P\big( \tau_\mu^{(1)} \in [ k_1, k_1 + 1]\big)\ \P\big( \tau_{H_1}^{(2)} \in [k_2,k_2 + 1]\big) \; \P\big(\tau_{H_2}^{(1)} < \infty\big)\Big] + \cE(\eps,M)
\end{align*}
where $\cE(\eps,M)$ contains the exponentially negligible terms.
Similarly, restricting the trajectory to the most probable combination $k_1$, $k_2$ gives a very close lower bound:
\begin{align*}
	\P(\gl_2 \geq \mu) &\geq \sup_{k_1, k_2 \in K} \Big[ \P\big(\tau^{(1)}_\mu \in [k_1, k_1+1]\big) \P\big(\tau^{(2)}_{H_1} \in [k_2, k_2+1]\big) \P\big( \tau^{(1)}_{H_2} < \infty\big) \Big]\,.
\end{align*}

By substituting the exponential estimates from Proposition \ref{propo:localhitting} and Proposition \ref{propo:global_hitting_prob}, up to polynomial factors in $\mu$, the probability scales exponentially as:
\begin{align*}
	\exp\Big(- \inf_{k_1,k_2 \in K}\big( \mu \, I_1(k_1/\mu) + H_1 \, I_2(k_2/H_1) + H_2 \, \frac{a+1}{2}\big)\Big)\,.
\end{align*}

Let $s_1 = k_1/\mu$ and $s_2 = k_2/H_1$ (note that $k_2 \leq M/\mu$ implies $s_2 \leq M \mu/H_1 \leq M$). Because the rate functions $I_1$ and $I_2$ are uniformly continuous on the compact domain $[\eps, M]$, evaluating this discrete infimum over the fine grid converges exactly to the continuum infimum over the real-valued relative times $s_1, s_2 \in [\eps, M]$. 

Furthermore, because the heights scale multiplicatively as $H_1 = \mu(1+s_1/4)$ and $H_2 = H_1(1+s_2/4)$, we can factor the height variables out of the exponent which gives for the minimal cost:
\begin{align*}
 \mu \inf_{s_1,s_2 \in [\eps,M]} \Big[ I_1(s_1) + \big(1+s_1/4\big) \inf_{s_2} \Big( I_2(s_2) + \big(1+s_2/4\big) \frac{a+1}{2} \Big) \Big]\,.
\end{align*}

For the probability $\{\gl_k \geq \mu\}$, applying the exact same grid discretization, leads to an exponential decay of the form $\exp(-\mu C_k)$, where the sequence of optimal costs $C_k$ is defined by the following recursion.

Let $V_n(H)$ be the cost starting at boundary height $H$ to have more than $n$ hitting times of the critical line. It is given by $V_n(H) = C_n H$ and we have the recursion relation:
\begin{align*}
	V_n(H) 
	&= H \min_t \Big[ \frac{\big(1 + (1/4 - a_n) t\big)^2}{2t} + C_{n-1} (1+  t/4)\Big] \,,
\end{align*}
where $a_n := -a/4$ if $n$ is odd, and $a_n := (a+1)/4$ if $n$ is even. Using the fact that the minimum of $t \mapsto A/t + B t$ when $A >0$ and $B>0$ is reached at $\sqrt{A/B}$ and its value is equal to $2 \sqrt{AB}$, this gives the recursive relation:
\begin{align}
	C_n = (1/4 - a_n) + C_{n-1} + \sqrt{(1/4 - a_n)^2 + C_{n-1}/2}\,. \label{eq:recurrence_C} 
\end{align}

The sequence $C_n$ admits a closed-form solution, which we prove by induction. 

\begin{lemma}[Explicit Exponential Costs]\label{lem:explicit_costs}
	For all $n \ge 1$, the sequence defined by \eqref{eq:recurrence_C} satisfies:
	\begin{align*}
		C_{2n-1} &= \frac{n(a+n)}{2}\,, \\
		C_{2n}   &= \frac{n(a+n+1)}{2}\,.
	\end{align*}
\end{lemma}

\begin{proof}
	We proceed by induction on $n$. We already know that $C_0 = 0$ and $C_1 = (a+1)/2$ from Proposition \ref{propo:global_hitting_prob}. 
Assume the formula holds for $2k-1$, then for the even step $2k$, the term inside the square root of \eqref{eq:recurrence_C} becomes:
	\begin{align*}
		\frac{a^2}{16} + \frac{C_{2k-1}}{2} 
		= \frac{(a+2k)^2}{16}\,.
	\end{align*}
	Substituting this back into the recurrence gives:
	\begin{align*}
		C_{2k} 
		= \frac{k(a+k+1)}{2}\,.
	\end{align*}
	Similarly, for the odd step $2k+1$, using $C_{2k} = \frac{k(a+k+1)}{2}$,
	\begin{align*}
		C_{2k+1} &= \frac{a+1}{4} + C_{2k} + \sqrt{\frac{(a+1)^2}{16} + \frac{C_{2k}}{2}}\,.
	\end{align*}
	We focus on the term inside the square root:
	\begin{align*}
		\frac{(a+1)^2}{16} + \frac{k(a+k+1)}{4} 
		&= \frac{(a + 2k + 1)^2}{16}\,.
	\end{align*}
	Taking the square root and adding the remaining terms yields:
	\begin{align*}
		C_{2k+1} =  \frac{(k+1)(a+k+1)}{2}\,.
	\end{align*}
	This concludes the induction.
\end{proof}

To conclude the proof of Proposition \ref{propo:asympkexplo}, we recall that observing $k$ points above the level $\mu$ strictly requires completing $2k-1$ alternating phases of the diffusions (starting and ending with the $-a/4$ drift phase hitting the critical line). Therefore, the asymptotic probability is precisely governed by $C_{2k-1}$:
\begin{align*}
	\lim_{\mu \to \infty} \frac{1}{\mu} \ln \P[\mathcal{M}_a[\mu,\infty) \geq k] = - C_{2k-1} = - \frac{k(a+k)}{2}\,.
\end{align*}

\subsection{Extension to $a \in (-1,0)$}
We briefly detail how the previous analysis naturally extends to the regime $a \in (-1,0)$. In this regime, recall that the drift of the $+$ phase is $a_2 = (a+1)/4 < 1/4$, meaning both the $-$ and $+$ phases are strictly slower than the critical line. 

Because the final $+$ phase no longer hits the critical line almost surely, observing $k$ points above the level $\mu$ requires evaluating the large deviation cost of all $2k$ phases. Let $\tilde{C}_n$ denote the optimal cost of a trajectory with $n$ phases ending in the $+$ phase. 

The sequence satisfies the exact same recurrence as \eqref{eq:recurrence_C}, but with shifted parity: the final phase is a $+$ phase, so $\tilde{a}_n = (a+1)/4$ for odd $n$, and $-a/4$ for even $n$. For $n=1$, substituting $\tilde{a}_1 = (a+1)/4$ and $\tilde{C}_0 = 0$ into the recurrence yields:
\begin{align*}
	\tilde{C}_1 = -\frac{a}{4} + \sqrt{\frac{a^2}{16}} = -\frac{a}{4} + \frac{|a|}{4} = -\frac{a}{2}\,,
\end{align*}
where we used the fact that $a < 0$. 

Iterating the recurrence \eqref{eq:recurrence_C}, an induction similar to Lemma \ref{lem:explicit_costs} yields the exact closed-form sequence for $n=2k$:
\begin{align*}
	\tilde{C}_{2k} = \frac{k^2 - ka}{2} = \frac{k(k + |a|)}{2}\,.
\end{align*}

Thus, the asymptotic probability is governed by:
\begin{align*}
	\lim_{\mu \to \infty} \frac{1}{\mu} \ln \P[\mathcal{M}_a[\mu,\infty) \geq k] = - \frac{k(k+|a|)}{2}\,.
\end{align*}
Notice that substituting $a \ge 0$ into our previous result $C_{2k-1} = k(k+a)/2$ gives the expression $k(k+|a|)/2$.

	\bibliographystyle{plain}
\bibliography{biblio1}

\end{document}